\documentclass[preprint,12pt]{elsarticle}

\usepackage{amsmath,amssymb,amsfonts,amsthm,amscd}
\usepackage{url}   
\def\div{\mathop{\mathrm{div}}\nolimits}
\def \R{\mathbb{R}}
\def \bRb{\mathbb{R}}

\newtheorem{thm}{Theorem}

\newdefinition{rmk}{Remark}
\newproof{pf}{Proof}
\newdefinition{proposition}{Proposition}
\newdefinition{assumption}{Assumption}

\journal{***}

\begin{document}

\begin{frontmatter}



\title{Stationary solutions of the Vlasov-Fokker-Planck equation: existence, characterization and phase-transition}


\author[MHD]{M.H. Duong\corref{cores}}\ead{m.h.duong@warwick.ac.uk}

\cortext[cores]{Corresponding author.}
\address[MHD]{Mathematics Institute, University of Warwick, Coventry CV4 7AL, UK.}
\author[JT]{J. Tugaut}
\address[JT]{Universit\'e Jean Monnet, Institut Camille Jordan, 23, rue du docteur Paul Michelon,
CS 82301,
42023 Saint-\'Etienne Cedex 2,
France.}
\ead{tugaut@math.cnrs.fr}
\begin{abstract}
In this paper, we study the set of stationary solutions of the Vlasov-Fokker-Planck (VFP) equation. This equation describes the time evolution of the probability distribution of a particle moving under the influence of a double-well potential, an interaction potential, a friction force and a stochastic force. We prove, under suitable assumptions, that the VFP equation does not have a unique stationary solution and that there exists a phase transition. Our study relies on the recent results by Tugaut and coauthors regarding the McKean-Vlasov equation.
\end{abstract}

\begin{keyword}


Invariant measure \sep Vlasov-Fokker-Planck equation \sep McKean-Vlasov equation \sep stochastic processes.

\MSC 60G10 \sep 35Q83  \sep 35Q84. 
\end{keyword}

\end{frontmatter}


\section{Introduction}
\label{sec:intro}

\subsection{The Vlasov-Fokker-Planck equation}

We consider the following \emph{Vlasov-Fokker-Planck (VFP)} equation,
\begin{equation}
\label{eq:VFP}
\partial_t \rho = - \div_q\Big(\rho \frac{p}{m}\Big) + \div_p\Big( \rho (\nabla_q V + \nabla_q \psi \ast \rho + \gamma \frac{p}{m})\Big)+\gamma kT\Delta_p \rho.
\end{equation}
In this equation, the spatial domain is $\mathbb{R}^{2d}$ with coordinates $(q,p)\in \mathbb{R}^d\times \mathbb{R}^d$. The unknown is a time-dependent probability measure $\rho\colon[0,T]\to \mathcal{P}(\mathbb{R}^{2d})$. Subscripts as in $\div_q$ and $\Delta_p$ indicate that the differential operators act only on those variables. The functions $V = V(q)$ and $\psi = \psi(q)$ are given. The convolution $\psi \ast \rho$ is defined by $(\psi\ast\rho)(q)=\int_{\mathbb{R}^{2d}}\psi(q-q')\rho(q',p')\,dq'dp'$. Finally $\gamma, k$ and $T$ are positive constants.

Equation~\eqref{eq:VFP} is the forward Kolmogorov equation of the following stochastic differential equation (SDE), 
\begin{align}
\label{SDE}
&dQ(t)= P(t)\, dt,\nonumber
\\&dP(t)=-\nabla V(Q(t))\,dt-\nabla\psi\ast \rho_t(Q(t))\,dt-\gamma\, \frac{P}{m}\,dt+\sqrt{2\gamma kT}\, dW(t),
\end{align}
where $\rho_t$ is the law of $(Q_t,P_t)$.
This SDE models the movement of a particle with mass $m$ under a fixed potential $V$, an interaction potential $\psi$,  a friction force (the drift term $-\gamma\frac{P}{m}\,dt$) and a stochastic forcing described by the $d$-dimensional Wiener measures~$W$. In this model, $\gamma$ is the friction coefficient, $k$ is the Boltzmann constant and $T$ is the absolute temperature.

Eq. \eqref{eq:VFP} and system \eqref{SDE} play an important role in applied sciences in particular in statistical mechanics. For instance, they are used as a simplified model for chemical reactions, or as a model for particles interacting through Coulomb, gravitational, or volume exclusion forces, see e.g., \cite{Kramers40, NPT84,BD95}. Eq. \eqref{eq:VFP} (and related models) has been studied intensively in the literature by many authors from various points of view, see e.g. \cite{Degond86,BD95,BGM10,DPZ13,DPZ14,Duo15} and references therein. In particular, invariant probabilities of Eq. \eqref{eq:VFP} have been investigated in \cite{Dre87,BGM10} (see also \cite{Duo15}). However, in these papers, the potential $V$ is assumed to be either bounded or globally Lipschitz or convex. As a result, there is a unique stationary solution. In this paper, we show that when the potential $V$ is unbounded, non-convex and not globally Lipschitz, of which a double-well potential is a typical example, non-uniqueness and phase transition can occur. Herein, we characterise the set of stationary solutions in such a case. Our study relies on the recent results by Tugaut and co-authors about the McKean-Vlasov diffusion by showing that the set of stationary solutions of the Vlasov-Fokker-Planck equation is related to that of the McKean-Vlasov equation.
\subsection{Normalization}
We first write \eqref{eq:VFP} in dimensionless form. The non-dimensionalization for Eq. \eqref{eq:VFP} has been done previously in the literature, see for instance \cite[Section 2.2.4]{LRS10}. For the sake of convenience, we perform it here. By setting
\begin{equation*}
q=:L\widetilde{q},\quad p=:\frac{mL}{\tau}\widetilde{p},\quad t=:\tau\widetilde{t}
\end{equation*}
and 
\begin{equation*}
V(q)=:\frac{mL^2}{\tau^2}\widetilde{V}(\widetilde{q}),\quad \psi(q)=:\frac{mL^2}{\tau^2}\widetilde{\psi}(\widetilde{q}), \quad \rho(p,q,t)=:\frac{\tau^d}{m^dL^{2d}}\widetilde{\rho}(\widetilde{p},\widetilde{q},\widetilde{t}),
\end{equation*}
where $L$ is the characteristic length scale, and $\tau:=\frac{m}{\gamma}$ is the relaxation time of the particle dynamics.
Then the dimensionless form of the Vlasov-Fokker-Planck equation is (after leaving out all the tilde)
\begin{equation}
\label{eq: equationdimensionles}
\partial_t \rho = - \div_q\Big(\rho p\Big) + \div_p\Big( \rho (\nabla_q V + \nabla_q \psi \ast \rho + p)\Big)+\lambda\Delta_p \rho.
\end{equation}
where $\lambda=kT\tau^2m^{-1}L^{-2}$ is the dimensionless diffusion coefficient.

In this paper, we are interested in stationary solutions of Eq. \eqref{eq: equationdimensionles}, i.e., solutions of the following equation
\begin{equation}
\label{eq: stationary equation}
\mathsf{K}[\rho](\rho)=0,
\end{equation}
where 
\begin{equation}
\mathsf{K}[\mu](\rho):=- \div_q\Big(\rho p\Big) + \div_p\Big( \rho (\nabla_q V + \nabla_q \psi \ast \mu + p)\Big)+\lambda\Delta_p \rho
\end{equation}
for given $\mu\in L^1(\R^{2d})$. Note that for a given $\mu$, the operator $\mathsf{K}[\mu](\rho)$ is linear in $\rho$.  This can be seen as a  linearised operator of $\mathsf{K}[\rho](\rho)$. 
\subsection{Organisation of the paper}
The rest of the paper is organised as follows. In Section \ref{sec: char}, we state our assumptions and provide a characterization via an implicit equation for a solution of Eq. \eqref{eq: stationary equation}. In Section \ref{sec: results} we present main results of the paper which prove the existence, (non-) uniqueness and phase transition properties of such stationary solutions.
\section{Characterization of invariant probabilities}
\label{sec: char}

In this section, we characterize solutions of Eq. \eqref{eq: stationary equation}.

First of all, we consider the following assumptions:
\begin{assumption}[Assumptions for the potential $V$]\ \\
The potential $V$ satisfies the following assumptions.
\label{assumption V}
\begin{enumerate}[(V1)]
\item $V$ is a smooth function and there exists $m\in\mathbb{N}^*$ and $C_{2m}>0$ such that $\lim_{|x|\to+\infty}\frac{V(x)}{|x|^{2m}}=C_{2m}$, where $|\cdot|$ denotes the Euclidean norm.
\item The equation $\nabla V(x)=0$ admits a finite number of solutions. We do not specify anything about the nature of these critical points. However, the local minima where the Hessian is positive will be denoted by $a_0$.
\item $V(x)\geq C_4|x|^4-C_2|x|^2$ for all $x\in\mathbb{R}^d$ with $C_2,C_4>0$. 
\item $\displaystyle\lim_{|x|\to\pm\infty}{\rm Hess}\,V(x)=+\infty$ and ${\rm Hess}\,V(x)>0$ for all $x\notin K$ where $K$ is a compact of $\mathbb{R}^d$ which contains all the critical points of $V$.
\end{enumerate}
\end{assumption}
\begin{assumption}[Assumptions for the interaction potential $\psi$]\ \\ The interaction potential $\psi$ satisfies the following assumptions.
\begin{enumerate}[($\psi$1)]
\item There exists an even polynomial function $G$ on $\mathbb{R}$ such that $\psi(x)=G(|x|)$. And, $\deg(G)=:2n\geq2$.
\item $G$ and $G''$ are convex.
\item $G(0)=0$.
\end{enumerate}
\end{assumption}
The simplest example (most famous in the literature) is that $V(x)=\frac{x^4}{4}-\frac{x^2}{2}$ (i.e., $V$ is a double-well potential) and $\psi(x)=\frac{\alpha}{2} x^2$ for some $\alpha$ (i.e., $\psi$ is a quadratic interaction).
\begin{proposition}
\label{prop: presentation of invariant}
Suppose that Assumption 1 and Assumption 2 hold. If there exists a solution $\rho_\infty\in L^1\cap L^\infty$ of Eq. \eqref{eq: stationary equation} then 
\begin{equation}
\label{eq: stationary measure}
\rho_{\infty}(q,p)=Z_\lambda^{-1}\exp \left[-\frac{1}{\lambda}\Big(\frac{p^2}{2}+V(q)+ \psi\ast\rho_\infty(q)\Big)\right],
\end{equation}
where $Z_\lambda$ is the normalizing constant
\begin{equation}
\label{eq: Z}
Z_\lambda=\int_{\mathbb{R}^{2d}}\exp \left[-\frac{1}{\lambda}\Big(\frac{p^2}{2}+V(q)+ \psi\ast\rho_\infty(q)\Big)\right]\,dq\,dp.
\end{equation}
Conversely any measure whose density satisfies \eqref{eq: stationary measure} is invariant for \eqref{eq: equationdimensionles}.
\end{proposition}
Let us remark that we use the convexity at infinity of $\psi$ in order $Z_\lambda$ to be finite and $\rho_\infty$ to be in $L^1\cap L^\infty$.
\begin{proof}
The idea of the proof has appeared in \cite{Dre87}, where the authors study the Vlasov-Fokker-Planck equation but with different scaling and assumptions. The proof is divided into two steps. 

\noindent\textbf{Step 1.} We first consider the linearised equation
\begin{equation}
\label{eq: linearised equation}
\mathsf{K}(\rho):=\mathsf{K}[\mu](\rho)=0,
\end{equation}
where $\mu\in L^1(\R^{2d})$ is given. We prove the following assertion: define
\begin{equation}
\label{eq: u stationary}
u(q,p):=C_\lambda^{-1}\exp\left(-\frac{1}{\lambda}\Big(\frac{1}{2}p^2+V(q)+\psi\ast \mu(q)\Big)\right),
\end{equation}
where $C_\lambda$ is the normalisation constant so that $\|u\|_{L^1}=1$,
and
\begin{align*}
A:=&\Big\{v:\R^{2d}\rightarrow \R\Big|v(\cdot,p)\in C^{1}(\R^d)\, \forall p\in \R^d; v(q,\cdot)\in C^{2}(\R^d)\, \forall q\in \R^d; ~\text{and}
\\& \quad f:=v\cdot u^{-1/2}~ \text{satisfies}~ f\in H^1((1+|p|+|h|)\,dqdp), \Delta_p f\in L^2\Big\},
\end{align*}
where $h(q):=\nabla_q V(q)+\nabla_q\psi\ast \mu$. Then $u$ is the unique solution in $A$ of the linearised equation \eqref{eq: linearised equation}.

Note that under the assumption that $V$ and $\psi$ are smooth, the linearised operator is hypo-elliptic, see for instance \cite{Hor67} and \cite[Section 1]{ DM10}. Hence we know a priori that all solutions of Eq.\eqref{eq: linearised equation} are smooth. Therefore, all the derivatives in this proof can be understood in the classical sense.

We now prove this assertion. By the assumptions on $V$ and $\psi$, $V(q)+\psi\ast\mu(q)$ behaves like a polynomial of order $k\geq 2$  of $|q|$ at infinity. This and \eqref{eq: u stationary} imply that
\begin{align*}
&\|u^\frac{1}{2}\|^2_{L^2((1+|p|+|h|)\,dqdp)}=\int_{\R^{2d}} u(1+|p|+|h|)\,dqdp<\infty,
\\&\|\nabla u^\frac{1}{2}\|^2_{L^2((1+|p|+|h|)\,dqdp)}=\left(\frac{1}{2\lambda}\right)^2\int_{\R^{2d}} u (|p|^2+|h|^2)(1+|p|+|h|)\,dqdp<\infty,
\\&\|\Delta_p u^\frac{1}{2}\|^2_{L^2}=\left(\frac{1}{2\lambda}\right)^2\int_{\R^{2d}} u \Big(d-\frac{1}{2\lambda}|p|^2\Big)^2\,dqdp<\infty.
\end{align*}
Therefore $u\in A$. Since $-\div_q(u p)+\div_p(u(\nabla_qV+\nabla_q\psi\ast\mu))=\div_p(u p)+\lambda \Delta_p u=0$, it follows that $\mathsf{K}[\mu](u)=0$. Now assume that Eq. \eqref{eq: linearised equation} has another solution $v \in A$ and $\|v\|_{L^1}=1$. Let $f:=v\cdot u^{-1/2}$. We have
\begin{align*}
&-\div_q(v p)+\div_p(v(\nabla_q V+\nabla_q\psi\ast\mu))\\
&=u^{1/2}[-\div_q(f p)+\div_p(f(\nabla_q V+\nabla_q\psi\ast \mu))],
\end{align*}
and
\begin{align*}
\div_p\Big(v p+ \lambda \nabla_pv\Big)&=\div_p\Big(v p+\lambda\nabla_p(u\,u^{-1/2}\,f)\Big)
\\&=\div_p\Big(v p+\lambda(u\nabla_p(u^{-1/2}f)+\nabla_pu\cdot u^{-1/2}f)\Big)
\\&=\lambda\div_p\Big(u\nabla_p(u^{-1/2}f)\Big).
\end{align*}
Define $\mathsf{Q}f:=-u^{-1/2}\mathsf{K}(u^{1/2} f)=-u^{-1/2}\mathsf{K}(v)$. Then from the above calculation, we get
\begin{equation*}
\mathsf{Q}f=-[-\div_q(f p)+\div_p(f(\nabla_q V+\nabla_q\psi\ast \mu))]-\lambda u^{-1/2}\div_p\Big(u\nabla_p(u^{-1/2}f)\Big).
\end{equation*}
Therefore, by multiplying by $f$ and integrating over $\R^{2d}$, we obtain
\begin{align*}
\langle \mathsf{Q}f,f\rangle_{L^2}&=\frac{1}{2}\int_{\R^{2d}}[\div_q(p f^2)-\div_p(f^2(\nabla_q V+\nabla_q\psi\ast \mu))]\,dqdp\\
&\qquad-\lambda\int_{\R^{2d}}u^{-1/2}\div_p\Big(u\nabla_p(u^{-1/2}f)\Big) f\,dqdp
\\&=\frac{1}{2}\int_{\R^{2d}}[\div_q(p f^2)-\div_p(f^2(\nabla_q V+\nabla_q\psi\ast \mu))]\,dqdp\\
&\qquad-\lambda\int_{\R^{2d}}\div_p\Big(u^{-1/2}f[u\nabla_p(u^{-1/2}f)]\Big)dqdp
\\&\qquad+\lambda\int_{\R^{2d}}u\Big(\nabla_p(u^{-1/2}f)\Big)^2\,dqdp\\
&=\lambda\int_{\R^{2d}}u\Big(\nabla_p(u^{-1/2}f)\Big)^2\,dqdp.
\end{align*}
Note that in the above computations, as we show in Remark \ref{re: remark} that due to divergence theorem and the fact that $ v\in A$, the first three integrals vanish.

Since $\mathsf{Q}f=0$, it follows that $\nabla_p(u^{-1/2}f)=0$, i.e., $u^{-1/2} f=g(q)$ for some function $g$. Hence $v=u^{1/2}f=u\cdot g(q)$, and $0=\mathsf{K}(v)=-up\cdot\nabla_qg(q)$. It implies that $\nabla_q g(q)=0$, i.e., $g$ is a constant. Since $\|v\|_{L^1}=1$, we obtain that $g=1$, i.e. $v=u$. In other words, Eq. \eqref{eq: linearised equation} has $u$ as a unique solution in $A$ and $\|u\|_1=1$.

\noindent\textbf{Step 2.} Suppose that $\rho_\infty\in L^1\cap L^\infty$ is a solution of Eq. \eqref{eq: stationary equation}. Therefore, $\rho_\infty$ solves the equation $\mathsf{K}[\rho_\infty](\nu)=0$. According to \textbf{Step 1}, this equation has a unique solution given by
\begin{equation*}
\tilde{\nu}=Z_\lambda^{-1}\exp \left[-\frac{1}{\lambda}\Big(\frac{p^2}{2}+V(q)+ \psi\ast\rho_\infty(q)\Big)\right].
\end{equation*}
Hence $\tilde{\nu}=\rho_\infty$, i.e., $\rho_\infty$ satisfies \eqref{eq: stationary measure}. The reverse assertion is obvious due to the convexity at infinity of $V$ and of $\psi$.
\end{proof}
\begin{rmk}
\label{re: remark}
We verify here that the first three terms in $\langle\mathsf{Q}f,f\rangle_{L^2}$ vanish. We present here for the first term only, since the computations for the other ones are similar. By definition of $A$, we have
\begin{align*}
&\|pf^2\|_{L^1}=\int_{\R^{2d}}|pf^2|\,dqdp\leq \|f\|^2_{L^2((1+|p|+|h|)\,dqdp)}<\infty,
\end{align*}
and 
\begin{align*}
\|\div_q(pf^2)\|_{L^1}&=2\int_{\R^{2d}}|f||p\cdot\nabla_q f|\,dqdp\leq 2\int_{\R^{2d}}|f||p||\nabla_q f|\,dqdp
\\&\leq 2\left(\int_{\R^2d}|f|^2|p|\,dqdp\right)^\frac{1}{2}\left(\int_{\R^2d}|\nabla_q f|^2|p|\,dqdp\right)^\frac{1}{2}
\\&\leq 2\|f\|_{L^2((1+|p|+|h|)\,dqdp)}\|\nabla f\|_{L^2((1+|p|+|h|)\,dqdp)}<\infty.
\end{align*}
Therefore, by the divergence theorem, see for instance \cite[Section 4.5.2]{GGS10}, we obtain $\int_{\R^2d}\div_q(pf^2)\,dqdp=0$, which is the desired equality.
\end{rmk}
\section{Main results}
\label{sec: results}
In this section, we assume that Assumption 1 and Assumption 2 are fulfilled.

\begin{thm}
\label{virginie}
We consider a measure $\rho_\infty$ on $\bRb^d\times\bRb^d$. If it is an invariant probability for \eqref{eq: equationdimensionles} then $q\mapsto\int_{\bRb^d}\rho_\infty(q,p)dp$ is an invariant probability of 
\begin{equation}
\label{McKean-Vlasov2}
dX(t)=-\nabla V(X(t))\,dt-\nabla\psi\ast \mu_t(X(t))\,dt+\sqrt{2\lambda}dW(t),
\end{equation}
where $\mu_t$ is the law of $X(t)$.
\end{thm}

\begin{proof}
Denote by $\hat{\rho}_\infty$ the first marginal of $\rho_\infty$, i.e., $\hat{\rho}_\infty(q)=\int_{\R^d}\rho_\infty(q,p)dp$. 
Suppose that $\rho_\infty$ is an invariant measure for \eqref{eq: equationdimensionles}. According to Proposition~\ref{prop: presentation of invariant}, $\rho_\infty$ satisfies \eqref{eq: stationary measure}, i.e.,
\begin{align}
\rho_\infty(q,p)&=\frac{\exp\left[-\frac{1}{\lambda}\Big(\frac{p^2}{2}+V(q)+\psi\ast\rho_\infty(q)\Big)\right]}{\int_{\R^{2d}}\exp\left[-\frac{1}{\lambda}\Big(\frac{p^2}{2}+V(q)+\psi\ast\rho_\infty(q)\Big)\right]dqdp}\nonumber\\
&=\frac{e^{-\frac{1}{\lambda}\frac{p^2}{2}}}{\int_{\R^d}e^{-\frac{1}{\lambda}\frac{p^2}{2}}\,dp}\times\frac{\exp\left[-\frac{1}{\lambda}\Big(V(q)+\psi\ast\rho_\infty(q)\Big)\right]}{\int_{\R^d}\exp\left[-\frac{1}{\lambda}\Big(V(q)+\psi\ast\rho_\infty(q)\Big)\right]dq}
\end{align}
It follows that
\begin{align*}
\hat{\rho}_\infty(q)&=\int_{\R^d}\rho_{\infty}(q,p)\,dp
\\&=\frac{\exp\left[-\frac{1}{\lambda}\Big(V(q)+\psi\ast\rho_\infty(q)\Big)\right]}{\int_{\R^d}\exp\left[-\frac{1}{\lambda}\Big(V(q)+\psi\ast\rho_\infty(q)\Big)\right]dq}.
\end{align*}

According to \cite[Lemma 2.2]{HT1} $\hat{\rho}_\infty$ is a stationary measure of the McKean-Vlasov SDE
\begin{equation}
\label{McKean-Vlasov}
dX(t)=-\nabla V(Q(t))\,dt-\nabla\psi\ast \mu_t(X(t))\,dt+\sqrt{2\lambda}dW(t),
\end{equation}
where $\mu_t$ is the law of $X(t)$. Indeed, the convexity at infinity of $V$ and of $\psi$ provides the sufficient conditions of integrability for $\hat{\rho}_\infty$ to be an invariant probability. This concludes the proof of this theorem. Note that the forward Kolmogorov equation associated to the McKean-Vlasov SDE is given by
\begin{equation}
\label{McKean-Vlasov PDE}
\partial_t\mu_t=\div[\mu_t(\nabla V+\nabla\psi\ast\mu_t)]+\lambda\Delta \mu_t.
\end{equation}
\end{proof}
Theorem \ref{virginie} establishes a one-to-one correspondence between invariant measures of the McKean-Vlasov equation \eqref{McKean-Vlasov PDE} and that of the Vlasov-Fokker-Planck equation \eqref{eq: equationdimensionles}. We therefore can deduce various results on the existence, (non)-uniqueness and characterization for the latter from the known results obtained by Tugaut and co-authors for the former.

The first result corresponds to the existence of an invariant probability under Assumptions 1 and 2.

\begin{proposition}
\label{thm:fr:subconv}
For any $\lambda>0$, there exists a stationary measure - that is a solution of Eq. \eqref{eq: stationary equation}.
\end{proposition}
This is a consequence of Proposition 3.1 in \cite{JOTP}. The idea is to consider the free-energy functional. By using some compactness argument, we are able to prove that a subsequence of the law of $X(t)$ admits an adherence value which is an invariant probability, where $X$ is the McKean-Vlasov diffusion. Then, it proves the existence of such an invariant probability.

Next results do describe the invariant probabilities.

First one deals with the case in which both $V$ and $\psi$ are even functions. In this case, by using Schauder fixed point theorem, we are able to obtain the existence of an invariant probability whose density with respect to the Lebesgue measure is even.

\begin{thm}
If both $V$ and $\psi$ are even, there exists a symmetric invariant probability (that is a measure whose density with respect to the Lebesgue measure is even).
\end{thm}

This is a consequence of Theorem 4.5 in \cite{HT1}.

Next proposition establishes that there is an invariant probability around each well, providing an additive condition on the potentials.

\begin{proposition}
\label{prop:gr:lin:exc}
Here, $d=1$. We assume that the interacting potential $\psi$ is quadratic: $\psi(x):=\frac{\alpha}{2}x^2$. Let $a_0$ be a critical point of $V$ such that $\alpha+V''(a_0)>0$ and
\begin{eqnarray}
\label{eq:gr:lin:epscrit}
\alpha>2\,\sup_{x\neq a_0}\frac{V(a_0)-V(x)}{\left(a_0-x\right)^2}\,.
\end{eqnarray}
Thus, for all $\delta\in]0\,;\,1[$, there exists $\lambda_0>0$ such that for all $\lambda\leq\lambda_0$, Diffusion \eqref{eq: equationdimensionles} admits an invariant probability $\rho_\infty$ satisfying
\begin{eqnarray*}
\left|\int_\bRb\int_\bRb q\rho_\infty(q,p)dqdp-a_0+\frac{V^{(3)}(a_0)}{4V''(a_0)\left(\alpha+V''(a_0)\right)}\,\lambda\right|\leq\delta\,\lambda\,.
\end{eqnarray*}
\end{proposition}

This is a consequence of Proposition 1.2 in \cite{PT}.

Let us explain briefly what Condition \eqref{eq:gr:lin:epscrit} means. When the temperature is small, each stationary measure which is not symmetric concentrates to a Dirac measure around a critical point. However, the well needs to be sufficiently deep in the following sense: for any $x\neq a_0$, one needs that $V(x)+\frac{\alpha}{2}\left(x-a_0\right)^2>V(a_0)$, which is exactly equivalent to Condition \eqref{eq:gr:lin:epscrit}.

In the following theorem, we take assumptions close to the ones in Proposition \ref{prop:gr:lin:exc}. Indeed, the interacting potential $\psi$ is assumed to be quadratic. However, by using some convexity assumptions on the derivatives of $V$, we will show that there are exactly one or three invariant probabilities. Moreover, we can simulate if we are in the uniqueness or in the thirdness case.

\begin{thm}
\label{thm:UT}
Here, $d=1$. We assume that
\begin{align}
\label{yamcha}
V(x)&=-\frac{\left|V''(0)\right|}{2}x^2+\sum_{p=2}^{q}\frac{\left|V^{(2p)}(0)\right|}{(2p)!}x^{2p}\,\,\,\mbox{with}\,\,\deg(V)=:2q\,.
\end{align}
And, $\psi(x):=\frac{\alpha}{2}x^2$. Thus, there exists $\lambda_c>0$ such that:
\begin{itemize}
 \item For all $\lambda\geq\lambda_c$, Diffusion \eqref{eq: equationdimensionles} admits a unique invariant probability, which is symmetric.
 \item For all $\lambda<\lambda_c$, Diffusion \eqref{eq: equationdimensionles} admits exactly three invariant probabilities.
\end{itemize}
Moreover, $\lambda_c$ is the unique solution of the equation, where  $z$ is the unknown,
\begin{eqnarray}
\label{eq:thm:UT}
\displaystyle\int_{\bRb_+} \left(4y^2-\frac{1}{2\alpha}\right)e^{\left(\left|V''(0)\right|-\alpha\right)4y^2-\sum_{p=2}^{q}\frac{2z^{p-1}\left|V^{(2p)}(0)\right|}{(2p)!}2^{2p}y^{2p}}dy=0\,.
\end{eqnarray}
\end{thm}

This is a consequence of Theorem 2.1 in \cite{PT}. The global idea is the following. The study of the invariant probabilities is equivalent to the study of the zeros of a function from $\mathbb{R}$ to $\mathbb{R}$. The behaviour of the function does depend on the value $\lambda$. If $\lambda\geq\lambda_c$, this function has a unique zero and if $\lambda<\lambda_c$, the function is increasing then decreasing on $\bRb_+$ so it admits three zeros.

A similar idea is used to obtain the following result.

\begin{proposition}
\label{prop:UT:grandepsilon}
Here, $d=1$. We assume that $\psi$ is quadratic: $\psi(x):=\frac{\alpha}{2}x^2$.\\
Thus, for any $\alpha\geq0$, there exists a critical value $\lambda_0(\alpha)$ such that Diffusion \eqref{eq: equationdimensionles} admits a unique invariant probability provided that $\lambda>\lambda_0(\alpha)$.
\end{proposition}

This is a consequence of Proposition 2.4 in \cite{PT}. In Proposition \ref{prop:UT:grandepsilon}, we provide a setting such that there is the uniqueness of the invariant probability. The main difference with Theorem \ref{thm:UT} is that we do not assume that the derivative of $V$ is convex.

\begin{thm}
\label{thm:mesures:excentrees}
Let $a_0$ be a point where $V$ admits a local minimum such that
\begin{eqnarray}
\label{eq:thm:mesures:excentrees}
V(x)+\psi(x-a_0)>V(a_0)\quad\mbox{for all}\quad x\neq a_0\,.
\end{eqnarray}
Then, for all $\kappa>0$ small enough, there exists $\lambda_0>0$ such that $\forall \lambda\in]0;\lambda_0[$, the diffusion \eqref{eq: equationdimensionles} admits a stationary measure $\rho_\infty$ satisfying
\begin{equation*}
\int_{\bRb^d}\int_{\mathbb{R}^d}\left|\left| q-a_0\right|\right|^{2n}\rho_\infty(q,p)dqdp\leq\kappa^{2n}\,.
\end{equation*}
\end{thm}

This is a consequence of Theorem 2.3 in \cite{JOTP}. The global idea is the same than the one of Proposition \ref{prop:gr:lin:exc} but the technical material is different. Indeed, we use the free-energy functional.

More generally, all the results in \cite{HT1,HT2,HT3,McKean1966,McKean,TT,CRAS,PT,JOTP} hold.

\section*{Acknowledgements}
We would like to thank the anonymous referee for his/her useful suggestions for improving the presentation of the paper.\\[2pt]
(M.H.D) \emph{The work of this paper has started when both the authors participated in the workshop ``Analytic approaches to scaling limits for random system'' held in the Hausdorff Research Institute for Mathematics (HIM) in January 2015. M.H. D would like to thank the HIM for supporting his stay at the HIM.}\\
(J.T.) \emph{I would like to thank Andr\'e Schlichting who invited me to the workshop where I have met M.H. Duong.}







%
%
%
\end{document}